\documentclass[12pt]{article}
\usepackage[utf8]{inputenc}
\usepackage{amssymb, amsmath,amsthm, thmtools, mathabx, mathtools}
\usepackage{bbm}
\usepackage{amsfonts}
\usepackage[left=2.5cm,right=2.5cm,top=2cm,bottom=2cm]{geometry}
\usepackage{xcolor}

\usepackage{tabto}
\TabPositions{2 cm, 4 cm, 6 cm, 8 cm}

\usepackage{tikz-cd}

\usepackage{tikz}
\usetikzlibrary{arrows, automata, positioning}

\usepackage{tocloft}
\usepackage{appendix}

\usepackage{enumerate}

\usepackage{hyperref}
\usepackage[capitalize]{cleveref}

\hypersetup{
  colorlinks   = true, 
  urlcolor     = blue, 
  linkcolor    = blue, 
  citecolor   = blue 
}
 
 \newtheorem{thmA}{Theorem}

\newtheorem{corA}[thmA]{Corollary}

\newtheorem{theorem}{Theorem}[section]
\newtheorem{corollary}[theorem]{Corollary}
\newtheorem{lemma}[theorem]{Lemma}

\newtheorem{proposition}[theorem]{Proposition}

\newtheorem*{theorem*}{Theorem}

\theoremstyle{definition}
\newtheorem{definition}[theorem]{Definition}
\newtheorem{example}[theorem]{Example}
\newtheorem{remark}[theorem]{Remark}

\newtheorem*{remark*}{Remark}
\newtheorem*{notation*}{Notation}
\newtheorem*{conventions*}{Conventions}
\newtheorem*{acks*}{Acknowledgements}
\newtheorem*{out*}{Outline}

\renewcommand\leq{\leqslant}
\renewcommand\geq{\geqslant}
\newcommand{\betti}{b^{(2)}}
\newcommand{\leuler}{\chi^{(2)}}

\newcommand{\N}{\mathbb{N}}
\newcommand{\Z}{\mathbb{Z}}
\newcommand{\Q}{\mathbb{Q}}
\newcommand{\R}{\mathbb{R}}
\newcommand{\C}{\mathbb{C}}

\newcommand{\btnmo}{b^{(2)}_{n-1}}
\newcommand{\bti}{b^{(2)}_i}
\newcommand{\btj}{b^{(2)}_j}
\newcommand{\btz}{b^{(2)}_0}
\newcommand{\btn}{b^{(2)}_n}
\newcommand{\bto}{b^{(2)}_1}
\newcommand{\btt}{b^{(2)}_2}
\newcommand{\bttt}{b^{(2)}_3}
\newcommand{\vN}{\mathcal{N}}

\newcommand{\id}{\mathrm{id}}

\newcommand{\bG}{\bar{G}}

\newcommand{\bP}{\bar{P}}

\newcommand{\normal}[1]{\left<\! \left< #1\right> \!\right>}
\newcommand{\colim}{\operatorname{colim}}

\begin{document}

\title{Kazhdan groups of dimension $16$ with \\ prescribed second $\ell^2$-Betti number}
\author{Francesco Fournier-Facio and Roman Sauer}
\date{\today}
\maketitle

\begin{abstract}
We construct a family of simple, lacunary hyperbolic groups with property $(T)$ that have rational cohomological dimension~$16$ and whose second $\ell^2$-Betti number can be prescribed to be any positive real. Moreover, we construct hyperbolic groups with property $(T)$ whose second $\ell^2$-Betti number can be prescribed to be any non-negative rational. Along the way, we present new constructions of measurably diverse finitely generated groups, and we prove that the second $\ell^2$-Betti number is far from being semi-continuous in the space of marked groups, even assuming good finiteness properties.
\end{abstract}


\section{Introduction}

In relation to the Atiyah Conjecture \cite{atiyah}, there have been many constructions of matrices over group rings of finitely generated groups whose kernels have irrational von Neumann dimension \cite{austin, grabowski1, grabowski2}. Much less is known about $\ell^2$-Betti numbers of groups: the main open question is whether there exists a group of type $FP_{n+1}(\Q)$ with irrational $n$-th $\ell^2$-Betti number \cite[Question 4]{grabowski1}, but even without the finiteness conditions, we only know of one construction for finitely generated groups. Using a construction of Gaboriau \cite[Proposition VI.16]{gaboriau:cost}, one can build finitely generated groups with an uncountable range of first $\ell^2$-Betti numbers \cite[Proposition 5.1]{ITD}.

We present a new construction of groups taking exotic values of second $\ell^2$-Betti numbers, and with several additional desirable properties. Given a set $\pi$ of primes, we say that a group $G$ has \emph{cohomological dimension $n$ modulo $\pi$}, if $G$ has cohomological dimension $n$ over every ring in which each prime from $\pi$ is invertible. Note that if $n< \infty$, this implies that $G$ only contains $\pi$-torsion. The \emph{$\ell^2$-Euler characteristic} of a group, when defined, is the alternating sum of its $\ell^2$-Betti numbers.

\begin{thmA}
\label{main}

There exists an integer $\chi \geq 1$ such that the following holds. Let $p \neq q$ be primes. There exists a family of groups $(G_x)_{x \in \R_{>0}}$ with the following properties.
\begin{enumerate}
    \item $G_x$ is simple, has property $(T)$, and is lacunary hyperbolic.
    \item $G_x$ has cohomological dimension $16$ modulo $\{p, q\}$.
    \item The second $\ell^2$-Betti number is $\btt(G_x) = x$. Further, $\betti_n(G_x)<\infty$ for every $n\in\N$ and every $x \in \R_{>0}$. 
    \item The $\ell^2$-Euler characteristic is $\leuler(G_x)= \btt(G_x)+\chi$.
\end{enumerate}
\end{thmA}

By renouncing simplicity, we can also construct groups where a single prime obstructs the finiteness of the cohomological dimension: see \cref{main simplified}. Property $(T)$ distinguishes these groups from the examples obtained from Gaboriau's construction, which arise as amalgamated products \cite[Proposition 5.1]{ITD}. Lacunary hyperbolicity distinguishes them from the counterexamples to the general Atiyah Conjecture, since these all contain lamplighters \cite[Problem 3]{grabowski2}, while lacunary hyperbolic groups cannot (\cref{lamplighters}).

Focusing on rational numbers, a simpler version of the construction leading to \cref{main} allows to build hyperbolic groups. 

\begin{thmA}
\label{main hyp}
	There exists an integer $\chi \geq 1$ such that the following holds. There exists a family of groups $(G_x)_{x \in \Q_{>0}}$ with the following properties.
\begin{enumerate}
    \item $G_x$ has property $(T)$ and is hyperbolic.
    \item $G_x$ has cohomological dimension $16$ modulo $\pi$, where $\pi$ is the set of prime factors of the denominator of~$x\in\Q_{>0}$.
    \item The second $\ell^2$-Betti number is $\btt(G_x) = x$.
    \item The $\ell^2$-Euler characteristic is $\leuler(G_x)= \btt(G_x)+\chi$.
\end{enumerate}
\end{thmA}

In particular, when $x \in \N$, the group $G_x$ constructed above has finite cohomological dimension over $\Z$, hence is torsion-free. This construction has the following consequence.

\begin{corA}
\label{pichot}

Let $(m_i)_{i \in \N}$ be a sequence of natural numbers. There exist property $(T)$ torsion-free hyperbolic groups $G, G_i$ of cohomological dimension $16$, with quotients $G \to G_i$ such that $G_i \xrightarrow{i \to \infty} G$ in the space of marked groups, and
    \[\btt(G_i) = m_i \text{ for all } i \in \N \text{ but } \btt(G) = 0.\]
\end{corA}

Therefore the second $\ell^2$-Betti number is far from being semi-continuous in the space of marked groups, even with the strong finiteness properties that are enjoyed by torsion-free hyperbolic groups. This is in contrast with the semi-continuity of the first $\ell^2$-Betti number \cite{pichot}.

\medskip

As noticed by Ioana--Tucker-Drob, it follows easily from Gaboriau's construction \cite[Proposition VI.16]{gaboriau:cost} that there exist uncountably many pairwise non-measure equivalent groups \cite[Proposition 5.1]{ITD}. However, in measurable group theory, property $(T)$ is especially desirable, and last year two constructions of measurably diverse families of property $(T)$ groups have appeared. One by the first author and Sun \cite{dimensions}, where these groups are distinguished by the whole sequence of their $\ell^2$-Betti numbers - in the same vein as the paper by L{\'{o}}pez Neumann~\cite{antonio}, which constructs an infinite family of finitely presented simple groups. Then very recently by Ioana--Tucker-Drob \cite{ITD}, where the groups have vanishing $\ell^2$-Betti numbers, and arise from the technology of wreath-like products \cite{wreathlike}. In particular, these groups have infinite rational cohomological dimension.

\begin{corA}
\label{ME}
    The groups from \cref{main} are pairwise non-measure equivalent.
\end{corA}

\begin{proof}
Let $G_x$ and $G_y$ be groups from the family. If they are measure equivalent with index~$c>0$, then $\btn(G_x)=c\btn(G_y)$ for all $n \in \N$ by Gaboriau's theorem~\cite{gaboriau}, and thus $\leuler(G_x)=c\leuler(G_y)$.
Hence
\[c\btt(G_y)+\chi=\btt(G_x)+\chi=\leuler(G_x)=c\leuler(G_y)=c\betti_2(G_y)+c\chi.\]
Since $\chi$ is a positive integer, this shows that $c=1$, and thus $\btt(G_x) = \btt(G_y)$, so $x = y$.
\end{proof}

The same argument gives, to the best of our knowledge, the first infinite family of measurably diverse hyperbolic groups with property $(T)$.

\begin{corA}
\label{ME hyp}
    The groups from \cref{main hyp} are pairwise non-measure equivalent. \qed
\end{corA}

Our computations of $\ell^2$-Betti numbers exploit the work of Petrosyan--Sun on group-theoretic Dehn filling \cite{PS2}. More recently \cite{PS3}, they applied this theory to compute $\ell^2$-Betti numbers, by working in convenient classes where the L{\"u}ck approximation and Atiyah conjectures are known to hold: virtually locally indicable and cocompactly cubulated groups. Note that these are incompatible with property $(T)$.

We instead take a different route, and start from a hyperbolic group enjoying the strong rigidity property $(T_2)$, introduced by Bader and the second author \cite{T2, ICM}. This group is a torsion-free cocompact lattice in the isometry group of the octonionic hyperbolic plane (\cref{t2 lattices}), which is the source of the number $16$ in our results. Under this condition, building on an observation from \cite{FFF}, we will see that the excision theorem from \cite{PS2} is sufficient to give exact computations of the second $\ell^2$-Betti number of iterated group-theoretic Dehn fillings. Let us mention that this construction could also be done in the setting of relatively hyperbolic groups, starting from non-cocompact lattices in the octonionic hyperbolic plane, thanks to the recent work \cite{LNP}; this would produce groups of cohomological dimension $15$ instead.

\medskip

All of the constructions mentioned above, both in this paper and in others, are non-elementary to various extents. Because of this, we feel it is worth presenting possibly the most elementary construction of finitely generated groups achieving uncountably many values of $\ell^2$-Betti numbers.

\begin{thmA}
\label{suspension}

    Let $p$ be a prime. There exists a family of finitely generated groups $(G_x)_{x \in \R_{> 0}}$ such that $\bto(G_x) = 0, \btt(G_x) = \infty$, and $\bttt(G_x) = x$. Moreover, $G_x$ is residually finite and has cohomological dimension $6$ modulo $p$.
\end{thmA}

Taking a free product with $\Z$ gives an uncountable family of pairwise non-measure equivalent groups (\cref{ME easy}). Of course, none of these groups have property $(T)$.

\begin{acks*}
FFF is supported by the Herchel Smith Postdoctoral Fellowship Fund. The authors thank Antonio L{\'o}pez Neumann and Bin Sun for comments on a previous version of this paper. They thank the Isaac Newton Institute for Mathematical Sciences, Cambridge, for support and hospitality during the programme Operators, Graphs, Groups, where work on this paper was undertaken. This work was supported by EPSRC grant EP/Z000580/1.
\end{acks*}

\section{Homology and cohomology}

We start by collecting some preliminary facts about homology and cohomology of groups. We fix a ring $R$ for the rest of this section. All rings in this paper are unital.

First is an easy application of the transfer, whose cohomological version is given in \cite[Proposition III.10.4 on p.~85]{brown}. The homological version can be obtained similarly via the transfer, alternatively it is a direct consequence of the Lyndon--Hochschild--Serre spectral sequence \cite[VII.6 on p.~171]{brown}.

\begin{proposition}
\label{coinvariants}

    Let $N < G$ a normal subgroup of finite index, and suppose that $[G:N]$ is invertible in $R$. Then for every $R[G]$-module $M$ and every $n \in \N$, the corestriction induces an isomorphism
    \[H_n(N; M)_{G/N} \cong H_n(G; M).\]
\end{proposition}

Next, we look at homology and cohomology of colimit groups, which will be especially relevant for the groups from \cref{main}.

\begin{proposition}[{\cite[Proposition 4.8 on p.~61]{bieri}}]
\label{colimit homology}

    Let $G_0 \to G_1 \to \cdots$ be a directed sequence of groups with colimit $G_\infty$, and let $M$ be an $R[G_\infty]$-module. Then there is an isomorphism
    \[H_n(G_\infty; M) \cong \colim H_n(G_i; M).\]
\end{proposition}

This is the standard fact that homology commutes with filtered colimits. This is of course not true in cohomology, but it does work under some additional assumptions. We point out that the following is only used to compute cohomological dimensions of the groups from Theorems \ref{main} and \ref{suspension}. If the reader is happy with the standard estimate of cohomological dimension as one more than the homological dimension \cite[Theorem 4.6 on p.~60]{bieri}, then they can use \cref{colimit homology}, and move to the next section.

\begin{proposition}
\label{colimit cohomology}

    Let $G_0 \to G_1 \to \cdots$ be a directed sequence of groups with colimit $G_\infty$, and let $M$ be an $R[G_\infty]$-module. Suppose that the pullback $H^n(G_{i+1}; M) \to H^n(G_i; M)$ is surjective for all $i \in \N$. Then
    \[H^{n+1}(G_\infty; M) \cong \varprojlim H^{n+1}(G_i; M).\]
\end{proposition}

\begin{proof}
    A tower of abelian groups
    \[\cdots \to D_2 \to D_1\]
    is said to satisfy the \emph{Mittag--Leffler condition} if for each $k$ there exists $j \geq k$ such that for all $i \geq j$ the image of $D_i \to D_k$ equals the image of $D_j \to D_k$ \cite[Definition 3.5.6 on p.~82]{weibel}. We will replace this with the stronger and easier to check condition that each map $D_{i+1} \to D_i$ is surjective.

    Let $X_0$ be a model of the classifying space $BG_0$\footnote{also denoted by $K(G_0, 1)$.}. Using mapping cones, we may inductively define $X_{i+1}$ as a model of $BG_{i+1}$ such that the map $X_i \to X_{i+1}$ is an inclusion of CW-complexes. The colimit~$X$, in this case a union, of the $X_i$ is a model of $BG_\infty$. 
    The universal covering $\widetilde X$, which is a free $G_\infty$-CW complex, is a model of $EG_\infty$. Let $\bar X_i$ be the restriction of the universal covering to $X_i$. By covering theory we have $\bar X_i\cong G_\infty\times_{G_i}\widetilde X_i$. 
    The cellular cochain complexes
    \[D_i = \hom_{R[G_\infty]}\bigl(C_*(\bar X_i), M\bigr)\cong \hom_{R[G_i]}\bigl(C_*(\widetilde X_i), M\bigr)\]
    compute the cohomology $H^*(G_i; M)$, and fit into a tower of cochain complexes. Since $\widetilde X$ is the colimit of the $G_\infty$-CW complexes $\bar X_i$, the projective limit of the $D_i$ is the cochain complex $\hom_{R[G_\infty]}\bigl(C_*(\widetilde X_i), M\bigr)$, which computes the cohomology of $G_\infty$. 
    
    By our choice of $X_i$, the maps $D_{i+1} \to D_i$ are degree-wise surjective because $C_\ast(\bar X_i)$ is degree-wise a direct $R[G_\infty]$-summand of $C_\ast(\widetilde X_i)$, so we can apply \cite[Theorem 3.5.8 on p.~83]{weibel}\footnote{the citation is about homology of chain complexes, hence the left-hand side term being one degree below instead of above} and obtain a short exact sequence
    \[0 \to \varprojlim\nolimits^{1} H^n(G_i; M) \to H^{n+1}(G_\infty; M) \to \varprojlim H^{n+1}(G_i; M) \to 0.\]
    By the assumption on surjectivity, the $\varprojlim\nolimits^{1}$ term vanishes \cite[Proposition 3.5.7 on p.~83]{weibel}, so we obtained the desired isomorphism.
\end{proof}

\begin{corollary}
\label{dimensions}

    Let $G_0 \to G_1 \to \cdots$ be a directed sequence of groups with colimit $G_\infty$. Suppose that the following hold for all $i \in \N$.
    \begin{itemize}
        \item For every $R[G_\infty]$-module $M$ the pullback $H^n(G_{i+1}; M) \to H^n(G_i; M)$ is surjective.
        \item $cd_R(G_i) = n$;
        \item $H^n(G_0; R) \neq 0$ and the pullback $H^{n-1}(G_{i+1}; R) \to H^{n-1}(G_i; R)$ is surjective.
    \end{itemize}
    Then $cd_R(G_\infty) = n$.
\end{corollary}

\begin{proof}
    By the assumption on surjectivity we can apply \cref{colimit cohomology}. First, for every $R[G_\infty]$-module $M$:
    \[H^{n+1}(G_\infty; M) \cong \varprojlim H^{n+1}(G_i; M) \cong \varprojlim 0 \cong 0,\]
    thus $cd_R(G_\infty) \leq n$. Second,
    \[H^n(G_\infty; R) \cong \varprojlim H^n(G_i; R) \twoheadrightarrow H^n(G_0; R) \neq 0,\]
    and so $cd_R(G_\infty) \geq n$.
\end{proof}

\section{\texorpdfstring{$\ell^2$}{}-Betti numbers of tracial unitary representations}
\label{section l2}

We refer the reader to~\cite{luck} for details and background, where everything is developed for group von Neumann algebras. L\"uck formulated his dimension theory for arbitrary finite von Neumann algebras in~\cite{lueck-dimension}. 

A unitary $G$-representation~$\pi$ on a Hilbert space~$V_\pi$ is called \emph{tracial} if the associated von Neumann algebra~$\vN(\pi) \coloneqq \pi(G)''$, which is the double commutant of~$\pi(G)$ in the algebra of bounded operators on~$V_\pi$, has a cyclic vector $\xi_\pi\in V_\pi$ such that
\[\tau(T) \coloneqq \langle T\xi_\pi, \xi_\pi\rangle_{V_\pi}\]
defines a finite trace on~$\vN(\pi)$. 
The embedding  
\begin{equation}\label{eq: rank dense embedding}
\vN(\pi)\hookrightarrow V_\pi : T\mapsto T\xi_\pi
\end{equation}
is isometric with respect to the norm $\Vert T\Vert^2=\tau(T^\ast T)$ on $\vN(\pi)$ and has dense image. The embedding is naturally a map of left $\vN(\pi)$-modules. 
It extends to an isometric isomorphism of the completion $L^2(\vN(\pi), \tau)\cong V_\pi$. 
Via this isomorphism, $V_\pi$ has a natural right $\vN(\pi)$-module structure that commutes with the left $\vN(\pi)$-module structure. On the dense set of vectors $T\xi_\pi$, the von Neumann algebra $\vN(\pi)$ acts from the right via $(T\xi_\pi)\cdot S=TS\xi_\pi$. 

We regard the cyclic vector, and therefore~$\tau$, as part of the data of a tracial unitary representation. 
Via the Gelfand--Naimark--Segal (GNS) construction, tracial unitary $G$-representations are in one-to-one correspondence with conjugation-invariant positive functions on~$G$~\cite[Theorem~C.4.10 on p.~354]{BHV}. 

\begin{example}\label{exa: tracial vN algebras}
The von Neumann algebra associated to the regular representation~$\lambda_G$ on~$\ell^2(G)$ with cyclic vector $1_G$ is the \emph{group von Neumann algebra} of~$G$, which we denote by~$\vN(G)$.

If $G \to Q$ is a surjective group homomorphism, then composing it with $\lambda_Q$ yields a tracial unitary $G$-representation on~$\ell^2(Q)$.
\end{example}

L\"uck introduced a dimension function   
\[ \dim_{\vN}\colon \left\{\text{$\vN$-modules}\right\}\to \R_{\ge 0}\cup\{\infty\}\]
for arbitrary (algebraic) modules over a finite von Neumann algebra~$\vN$ with finite normalised trace~$\tau$. The choice of $\tau$ is understood to be specified, and we usually drop~$\tau$ from the notation. L\"uck's dimension function is additive for short exact sequences of $\vN$-modules and normalised, i.e.~$\dim_{\vN}(\vN)=1$. Moreover:

\begin{lemma}[{\cite[Theorem 0.6 (b) and (c)]{lueck-dimension}}]
\label{direct sums}
    Let $(V_i)_{i \in \N}$ be a sequence of $\vN$-modules. Then
    \[\dim_\vN \left( \bigoplus_{i \in \N} V_i \right) = \sum\limits_{i \in \N} \dim_{\vN}(V_i).\]
\end{lemma}

Let $X$ be a topological $G$-space. We do not require a cellular structure on~$X$ nor freeness of the action. Let $\pi$ be a tracial unitary $G$-representation. 
The equivariant homology 
\[H_n^G\bigl(X, \vN(\pi)\bigr)=H_n\bigl(\vN(\pi)\otimes_{\Z[G]}C_\ast(X)\bigr)\]
is naturally a left $\vN(\pi)$-module. Here $C_\ast(X)$ denotes the singular chain complex of~$X$. We have 
\[H_n(G; \vN(\pi)) = H_n^G(EG; \vN(\pi)).\]

\begin{definition}
Let $\pi$ be a tracial unitary $G$-representation. The \emph{$\ell^2$-Betti numbers of~$\pi$} are defined as 
\[ \betti_n(G;\pi)=\dim_{\vN(\pi)}H_n\bigl(G; \vN(\pi)\bigr).\]
When $\pi$ is the regular representation, we simply write $\btn(G)$. When $G \to Q$ is a quotient, and $\pi$ is the pullback of the regular representation of $Q$, we write $\btn(G \to Q)$.
\end{definition}


\begin{definition}
Let $\pi$ be a tracial unitary $G$-representation such that 
$\sum_{n \in \N} \betti_n(G;\pi)<\infty$. Then we define the \emph{$\ell^2$-Euler characteristic} of~$\pi$ as 
\[ \leuler(G;\pi)=\sum_{n \in \N}(-1)^n\betti_n(G;\pi).\]
When $\pi$ is the regular representation we just write $\leuler(G)$.
\end{definition}    

The next result generalises the Euler--Poincar{\'e} formula for the regular representation to arbitrary tracial unitary representations. 

\begin{proposition}\label{Euler Poincare formula}
Let $G$ be a group of type~$FP(\Q)$. Then the Euler characteristic of $G$ and the $\ell^2$-Euler characteristic of every tracial unitary $G$-representation~$\pi$ coincide:
\[ \chi(G)=\leuler(G;\pi).\]
\end{proposition}

\begin{proof}
Let $P_\ast$ be a finite projective $\Q[G]$-resolution of the trivial $\Q[G]$-module~$\Q$. Then $C_\ast=\vN(\pi)\otimes_{\Q[G]}P_\ast$ is a finite projective $\vN(\pi)$-chain complex. 
The proof that 
\[ \sum_{n\ge 0}(-1)^n\dim_{\vN(\pi)}(C_n)=\sum_{n\ge 0}(-1)^n\dim_{\vN(\pi)}\bigl(H_n(C_\ast)\bigr)\]
is the same as the Euler--Poincar{\'e} formula for the classical Euler characteristic and basically follows from the additivity of the dimension.  See the standard proof in Hatcher's book~\cite[Theorem~2.44 on p.~146]{hatcher}.
\end{proof}

The following result is an analogue of~\cite[Theorem~2.2]{peterson+thom}, which treated the cohomological version, and only for the regular representation. 
According to the discussion after~\cref{eq: rank dense embedding}, the Hilbert space $V_\pi$ carries a left $\vN(\pi)$-module structure and a right $\vN(\pi)$-module structure, in particular, also a right $\C[G]$-module structure. So the homology on the right hand side is naturally a left $\vN(\pi)$-module.

\begin{proposition}
\label{l2 vs vN}
Let $\pi$ be a tracial unitary $G$-representation on the Hilbert space~$V_\pi$. Then we have 
\[ \betti_n(G;\pi)=\dim_{\vN(\pi)} H_n\bigl(G;V_\pi\bigr)\]
for every $n\in\N$. 
\end{proposition}

\begin{proof}
 The embedding  $\vN(\pi)\hookrightarrow V_\pi$ of from~\cref{eq: rank dense embedding} is rank dense: 
  
 One calls a $\vN(\pi)$-linear homomorphism $f\colon M\to N$ of $\vN(\pi)$-modules \emph{rank dense} if for every $x\in N$ there is an increasing sequence of projections $p_i\in \vN(\pi), i \in \N$, converging to $\id$ such that $p_ix\in f(M)$. 
By~\cite[Lemma~A.5]{kyed+petersen+vaes}
the embedding $\vN(\pi)\hookrightarrow V_\pi$ is rank dense because it is dense. 

The embedding $\vN(\pi)\hookrightarrow V_\pi$ is also right $\vN(\pi)$-linear, in particular, an embedding of right $\C[G]$-modules, and the right $\C[G]$-action commutes with the left $\vN(\pi)$-action. 
Now we consider the quotient $W \coloneqq V_\pi/\vN(\pi)$; this is an undesirable object from a functional-analytic perspective, however it makes sense to look at $W$ as an (algebraic) $\vN(\pi)$-module. 
By rank density for every $x\in W$ there is an increasing sequence of projections $p_i\in\vN(\pi)$ converging to~$\id$ such that $p_ix=0$ for every $i\in\N$. 
By the local criterion~\cite[Theorem~2.4]{sauer-phd} we obtain that $\dim_{\vN(\pi)}(W)=0$.
Let $P_\ast$ be a free $\C[G]$-resolution of~$\C$ of finite type. We have a short exact sequence of $\vN(\pi)$-chain complexes 
\[ 0\to \vN(\pi) \otimes_{\C[G]} P_\ast \to V_\pi \otimes_{\C[G]} P_\ast \to W \otimes_{\C[G]} P_\ast \to 0.\]
By \cref{direct sums}, the right hand side is a chain complex of modules of $\dim_{\vN(\pi)}$-dimension zero. By additivity and the long exact sequence in homology, the statement follows.
\end{proof}

\cref{l2 vs vN} will be especially useful in conjunction with higher property $(T)$.

\begin{definition}[\cite{T2}]
    We say that a group $G$ has \emph{property $[T_n]$} if for every $i \leq n$ and for every unitary representation $V$, it holds $H^i(G; V) = 0$. We say that $G$ has \emph{property $(T_n)$} if the vanishing holds under the additional assumption that $V^G = 0$.
\end{definition}

Since we will mostly be working with homology, the following will be more useful.

\begin{proposition}[\cite{ICM}]
\label{tn homology}

    Let $G$ be a group with property $[T_n]$ (respectively, $(T_n)$) and of type $FP_{n+1}(\Q)$. Then for every $i \leq n$ and for every unitary representation $V$ (respectively, every unitary representation $V$ with $V^G = 0$), it holds $H_i(G; V) = 0$.
\end{proposition}

\begin{proof}
    Combine \cite[Lemma 29]{ICM} and \cite[Lemma 31]{ICM}.
\end{proof}

\begin{corollary}
\label{tn betti quotient}

    Let $G$ be a group with property $[T_n]$ (respectively, $(T_n)$) and of type $FP(\Q)$, and let $Q$ be a quotient (respectively, an infinite quotient) of $G$. Then for every $i \leq n$ it holds $\bti(G \to Q) = 0$. Hence
    \[\leuler(G) = \sum\limits_{i > n} (-1)^i\bti(G \to Q).\]
\end{corollary}

We recall that hyperbolic groups are of type $FP(\Q)$ \cite[Theorem III.$\Gamma$.3.21 on p~468]{bridsonhaefliger}.

\begin{proof}
    Combine Propositions~\ref{l2 vs vN} and~\ref{tn homology}. The second statement now follows from \cref{Euler Poincare formula}.
\end{proof}

Our starting point will be the following example.

\begin{theorem}[\cite{T2}]
\label{t2 lattices}
    Let $G$ be a cocompact lattice in the isometry group of the octonionic hyperbolic plane. Then $G$ has properties $(T_3)$ and $[T_2]$.
\end{theorem}

Cocompact lattices in the isometry group of the octonionic hyperbolic plane exist \cite{BHC}, and they can be chosen torsion-free up to passing to a finite index subgroup \cite{selberg}. In the proofs of our main results, we will actually only need that $G$ has property $(T_2)$: see \cref{t2 sufficient}.

\begin{proof}
    The ambient Lie group has property $(T_3)$ by \cite[Theorem B and Appendix A]{T2}. We emphasise that this relies only on a small part of~\cite{T2}: It follows from a result of Kumaresan~\cite[Theorem~3]{kumaresan}
    for non-trivial irreducible unitary representation of the ambient Lie group, and the ultrapower technique in~\cite{T2}.
    Finally, property~$(T_3)$ passes to cocompact lattices by~\cite[Lemma 3.13]{T2}. Moreover, the second Betti number of $G$ vanishes, because it is the fundamental groups of a non-Hermitian locally symmetric space of rank one \cite{matsushima}, so $G$ also has $[T_2]$.
\end{proof}

We conclude with a useful fact about subgroups.

\begin{proposition}[{\cite[Theorem 6.29 on p.~253]{luck}}]
\label{subgroup induction}

    Let $G \to Q$ be a homomorphism, and let $Q \to L$ be an injective homomorphism. Then
    \[\dim_{\vN(Q)}H_n(G; \vN(Q)) = \dim_{\vN(L)} H_n(G; \vN(L)),\]
    for all $n \in \N$.
\end{proposition}

\section{Group-theoretic Dehn filling}

In this section we introduce group-theoretic Dehn filling, and do some preliminary work towards our main construction.

\subsection{Cohen--Lyndon triples}

Let $P < G$ and let $N$ be a normal subgroup of $P$. We denote $\bP \coloneqq P/N$ and $\bG \coloneqq G/\normal{N}$, where $\normal{N} = \normal{N}_G$ is the normal closure of $N$ in $G$. Such a quotient is called a \emph{group-theoretic Dehn filling of $(G, P)$}, or simply of $G$, if $P$ is clear from the context.

\begin{definition}
    We say that $(G, P, N)$ is a \emph{Cohen--Lyndon triple} if there exists a set of left transversals $T$ of $G/P \normal{N}$ such that
    \[\normal{N} = \ast_{t \in T} tNt^{-1}.\]
\end{definition}

This notion goes back to the work of Cohen--Lyndon on one-relator groups \cite{CL}; in this form it is defined in \cite{PS1}. The most important consequence is the following algebraic excision theorem.

\begin{theorem}[{\cite[Theorem A (ii)]{PS2}, see also \cite[Corollary 4.5]{PS3}}]
\label{excision}
    Let $(G, P, N)$ be a Cohen--Lyndon triple. Then for all $\Z [\bG]$-modules $M$ and all $n \in \N$, the quotients $G \to \bG$ and $P \to \bP$ induce isomorphisms
    \[H_n(G, P; M) \cong H_n(\bG, \bP; M) \quad \text{ and } \quad H^n(G, P; M) \cong H^n(\bG, \bP; M).\]
\end{theorem}

We will mostly be concerned with an especially well-behaved case.

\begin{corollary}
\label{excision elementary}
    Let $(G, P, N)$ be a Cohen--Lyndon triple. Suppose that $P$ is virtually cyclic, $N$ is infinite cyclic, and hence $\bP$ is finite. Let $R$ be a ring where $[P:N]$ is invertible. Then for all $R [\bG]$-modules $M$ and all $n \geq 3$, the quotient $G \to \bG$ induces isomorphisms
    \[H_n(G; M) \cong H_n(\bG; M) \quad \text{ and } \quad H^n(G; M) \cong H^n(\bG; M).\]
    Moreover, there is an exact sequence
    \[0 \to H_2(G; M) \to H_2(\bG; M) \to M_{\bP} \to H_1(G;M).\]
\end{corollary}

\begin{proof}
    By the assumption on invertibility of $[P:N]$, we have $cd_R(\bP) = 0$ and $cd_R(P) = 1$. Therefore $H_n(\bG, \bP; M) \cong H_n(\bG; M)$ for all $n \geq 2$, and $H_n(G, P; M) \cong H_n(G; M)$ for all $n \geq 3$. Using this, for all $n \geq 3$, the long exact sequences in homology for the pairs $(\bG, \bP)$ and $(G, P)$, together with \cref{excision}, give isomorphisms
    \[H_n(\bG; M) \cong H_n(\bG, \bP; M) \cong H_n(G, P; M) \cong H_n(G; M),\]
    and similarly for cohomology.

    For the second statement, the long exact sequence in homology of the pair $(G, P)$ includes:
    \[H_2(P; M) \to H_2(G; M) \to H_2(G, P; M) \to H_1(P;M) \to H_1(G;M).\]
    The first term vanishes because $cd_R(P) = 1$. Using $cd_R(\bP) = 0$ and \cref{excision}, the third term becomes $H_2(\bG; M)$. Finally, the fourth term can be replaced by $H_1(N; M)_{\bP}$ by \cref{coinvariants}. But $N \cong \Z$ acts trivially on $M$, so $H_1(N; M) = M$.
\end{proof}



A useful feature of the Cohen--Lyndon property is that it is transitive. The following is a generalisation of \cite[Lemma 4.22]{wreathlike2}, which deals with the case in which $L = N$.

\begin{lemma}
\label{CL transitive}
Let $N < L < P < G$ be groups such that $(G, P, \normal{N}_P)$ and $(P, L, N)$ are Cohen--Lyndon triples (in particular $N$ is normal in $L$). Then $(G, L, N)$ is a Cohen--Lyndon triple.
\end{lemma}

\begin{proof}
Let $T$ be a set of left transversals of $G/P\normal{N}_G$ such that $\normal{N}_G = *_{t \in T} t \normal{N}_P t^{-1}$, and let $S$ be a set of left transversals of $P/L \normal{N}_P$ such that $\normal{N}_P = *_{s \in S} s N s^{-1}$. Hence
\[\normal{N}_G = *_{t \in T} t \normal{N}_P t^{-1} = *_{t \in T, s \in S} t s N (t s)^{-1},\]
so it remains to show that $TS$ is a set of left transversals for $G/L\normal{N}_G$.

For $g \in G$, there exist $t \in T$ and $p \in P$ such that $g \in t p \normal{N}_G$. Further, there exists $s \in S$ such that $p \in s L \normal{N}_P$. Hence
\[g \in t p \normal{N}_G \subset t (s L \normal{N}_P) \normal{N}_G = ts L \normal{N}_G.\]
Suppose that there exist $t_i \in T, s_i \in S$ such that $t_1 s_1 \in t_2 s_2 L \normal{N}_G$. Then
\[t_2^{-1} t_1 \in s_2 L \normal{N}_G s_1^{-1} = s_2 L s_1^{-1} \normal{N}_G \subset P \normal{N}_G.\]
Because $T$ is a set of left transversals of $G/P \normal{N}_G$, we have $t_1 = t_2$, so $s_2^{-1} s_1 \in L \normal{N}_G$. Let $l \in L$ be such that $l^{-1} s_2^{-1} s_1 \in \normal{N}_G$. Now $l^{-1} s_2^{-1} s_1 \in P$, and $\normal{N}_G \cap P = \normal{N}_P$ because $(G, P, \normal{N}_P)$ is a Cohen--Lyndon triple \cite[Proposition 6.1(a)]{PS1}, so $l^{-1} s_2^{-1} s_1 \in \normal{N}_P$, that is $s_1 \in s_2 L \normal{N}_P$. Because $S$ is a set of left transversals for $P/L \normal{N}_P$, we have $s_1 = s_2$ and conclude.
\end{proof}

\subsection{(Relatively) hyperbolic groups}

When $G$ is hyperbolic relative to $P$, group-theoretic Dehn fillings are especially well-behaved \cite{osin}. The following was proved in \cite{CL:RH}, and then in the more general context of groups with hyperbolically embedded subgroups \cite{PS1}.

\begin{theorem}[{\cite{CL:RH, PS1}}]
\label{CL RH}

    Let $G$ be hyperbolic relative to $P$. Then for every sufficiently deep normal subgroup $N < P$, the triple $(G, P, N)$ is Cohen--Lyndon.
\end{theorem}

Here a property is said to hold for \emph{sufficiently deep} normal subgroups $N < P$ if there exists a finite subset $F \subset P \setminus \{ 1 \}$ such that the property holds for all normal subgroups $N < P$ such that $N \cap F = \emptyset$. In this context, group-theoretic Dehn fillings retain many geometric properties.

\begin{theorem}[\cite{osin}]
\label{osin}

    Let $G$ be hyperbolic relative to $P$. Let $B \subset G$ be a finite subset. Then for every sufficiently deep normal subgroup $N < P$, denoting by $\bG \coloneqq G/\normal{N}$ and $\bP \coloneqq P/N$:
    \begin{enumerate}
        \item The group $\bG$ is hyperbolic relative to $\bP$;
        \item The quotient $G \to \bG$ is injective on $B$ and induces embedding $\bP \to \bG$.
    \end{enumerate}
\end{theorem}

We will mostly be concerned with a particular example of a relatively hyperbolic pair. Recall that a subgroup $P < G$ is \emph{almost malnormal} if $gPg^{-1} \cap P$ is finite for all $g \in G \setminus P$.

\begin{theorem}[{\cite[Theorem 7.11]{bowditch}}]
\label{quasiconvex RH}

    Let $G$ be a non-elementary hyperbolic group, and let $P < G$ be quasiconvex and almost malnormal. Then $G$ is hyperbolic relative to $P$.
\end{theorem}

The main example is the following.

\begin{definition}
Let $G$ be a non-elementary hyperbolic group, and $g \in G$ an element of infinite order. The \emph{elementary closure} of $g$ is the unique maximal virtually cyclic subgroup $E_G(g) < G$ containing $g$. We say that $g \in G$ is \emph{special} if it has infinite order and $E_G(g) = \langle g \rangle$.
\end{definition}

Starting from a quasiconvex almost malnormal subgroup, we can build a new one. The following notion is central in much of the literature on small cancellation theory on groups with hyperbolic features \cite{olshanskii, osin:sc, hull}.

\begin{definition}
Let $G$ be a non-elementary hyperbolic group. A subgroup $H < G$ is \emph{suitable} if it is non-elementary and does not normalise any non-trivial finite subgroup of $G$.
\end{definition}

Once again there is a main example we will care about.

\begin{lemma}[{\cite[Lemma 2.4]{osin:measures}}]
\label{normal suitable}

Let $G$ be a non-elementary hyperbolic group with no non-trivial finite normal subgroups. Then every non-trivial normal subgroup of $G$ is suitable.
\end{lemma}

The following, which is the consequence of the more general \cite[Theorem 6.2]{abbotthull}, says that we can extend quasiconvex almost malnormal subgroups via elements of suitable subgroups.

\begin{theorem}[\cite{abbotthull}]
\label{randomwalk}

	Let $G$ be a non-elementary hyperbolic group with no non-trivial finite normal subgroups, and let $P < G$ be an infinite-index quasiconvex almost malnormal subgroup. Let $H < G$ be a suitable subgroup. Then there exists an element $h \in H$ such that $\langle P, h\rangle \cong P * \langle h \rangle$ is an infinite-index quasiconvex almost malnormal subgroup.
\end{theorem}

We can now use these results to prove the two main technical propositions of this section, which will be used as the induction step in our constructions. The first is about Dehn filling with respect to a \emph{large power} of an \emph{arbitrary} element.

\begin{proposition}
\label{induction step large power}

    Let $G$ be a non-elementary hyperbolic group with no non-trivial finite normal subgroups, and let $B \subset G$ be a finite subset. Let $g \in G$ be an element of infinite order and let $P \coloneqq E_G(g)$. Then for every sufficiently deep finite-index normal subgroups $N < P$, denoting by $\bG \coloneqq G/\normal{N}$ and $\bP \coloneqq P/N$:
    \begin{enumerate}
        \item The triple $(G, P, N)$ is Cohen--Lyndon;
        \item The group $\bG$ is non-elementary hyperbolic with no non-trivial finite normal subgroups;
        \item The quotient $G \to \bG$ is injective on $B$, and induces an embedding $\bP \hookrightarrow \bG$.
    \end{enumerate}
\end{proposition}

\begin{proof}
    By \cref{quasiconvex RH}, $G$ is hyperbolic relative to $P$. The Cohen--Lyndon property follows from \cref{CL RH}; the hyperbolicity follows from \cref{osin}, since a group that is hyperbolic relative to a finite subgroup is hyperbolic; the injectivity follows again from \cref{osin}.

    It remains to prove that $\bG$ is non-elementary and without finite normal subgroups. To this end, we use \cref{randomwalk} to find an additional element $h \in G$ such that $G$ is hyperbolic relative to $P \coloneqq E_G(g) * \langle h \rangle$. We can now see $G/\normal{N}_G$ as a Dehn filling of $(G, P)$ with respect to the sufficiently deep normal subgroup $\normal{N}_P < P$. In particular, $\bG$ is hyperbolic relative to $P/\normal{N}_P \cong \bP * \langle h \rangle$. Since this is not virtually cyclic and has no non-trivial finite normal subgroups, the same is true for $G$, by \cite[Theorem 6.14(b)]{DGO}.
\end{proof}

The second result is about Dehn filling with respect to a \emph{prescribed power} of a \emph{well-chosen} element.

\begin{proposition}
\label{induction step prescribed power}

    Let $G$ be a non-elementary hyperbolic group with no non-trivial finite normal subgroups, let $B \subset G$ be a finite subset and let $k \geq 1$. Let $H < G$ be a suitable subgroup and let $x \in G$. Then there exists an element $g \in xH$ with the following properties. Let $\bG \coloneqq G/\normal{g^k}$.
    \begin{enumerate}
        \item The triple $(G, \langle g \rangle, \langle g^k \rangle)$ is Cohen--Lyndon;
        \item The group $\bG$ is non-elementary hyperbolic with no non-trivial finite normal subgroups;
        \item The quotient $G \to \bG$ is injective on $B$, and the image of $g$ has order $k$.
    \end{enumerate}
\end{proposition}

\begin{proof}
We first find a special element $y \in xH$. By \cite[Lemma 3.8]{SQ}, the suitable subgroup $H$ contains a special element $h$. If $x \in H$, then we can set $y = h \in xH$; otherwise \cite[Lemma 3.2]{osinthom} ensures that $x\langle h \rangle$ contains a special element $y$. So $E_G(y) = \langle y \rangle$ is infinite-index, quasiconvex and almost malnormal in $G$. Applying \cref{randomwalk} twice, we find elements $z, w \in H$ such that $P \coloneqq \langle y, z, w \rangle$ is a free group of rank $3$ that is infinite-index, quasiconvex and almost malnormal. By \cref{quasiconvex RH}, $G$ is hyperbolic relative to $P$. We choose
\[g \coloneqq y z w z^2 w^2 \cdots z^n w^n,\]
where $n$ is large enough such that $g$ is a $C'(\lambda)$-small cancellation word in $P$, for some small parameter $\lambda > 0$. By classical small cancellation theory (see e.g. \cite[Chapter V]{lyndon:schupp}), the normal subgroup $\normal{g}_P < P$ is sufficiently deep, hence so is $\normal{g^k}_P$. We can thus apply \cref{osin} and obtain that $G / \normal{g^k}_G$ is hyperbolic relative to $P/\normal{g^k}_P$, and that both $B$ and $P/\normal{g^k}_P$ embed into $\bG$.

If $k = 1$, then $P/\normal{g^k}_P = P/\normal{g}_P$ is a classical small cancellation group, hence it is hyperbolic \cite{gromov:hyp}. If $k > 1$, then $P/\normal{g^k}_P$ is a one-relator group with torsion, hence Newman's spelling theorem \cite{newman} implies that it is hyperbolic, and that the image of $g$ has order $k$. In both cases, $\bG$ is hyperbolic relative to a non-elementary hyperbolic group, hence it is itself non-elementary hyperbolic. Since $P/\normal{g^k}_P$ also has no non-trivial finite normal subgroups in both cases, the same is true of $\bG$ by \cite[Theorem 6.14(b)]{DGO}. Finally, in both cases $g$ has order $k$ in $P/\normal{g^k}_P$, which embeds in $\bG$, so $g$ has order $k$ in $\bG$ as well.

It remains to show the first item. Since $\normal{g^k}_P < P$ is sufficiently deep, by \cref{CL RH} the triple $(G, P, \normal{g^k}_P)$ is Cohen--Lyndon. Moreover by the Cohen--Lyndon theorem on one-relator groups \cite{CL}, the triple $(P, \langle g \rangle, \langle g^k \rangle)$ is Cohen--Lyndon. So \cref{CL transitive} implies that $(G, \langle g \rangle, \langle g^k \rangle)$ is Cohen--Lyndon.
\end{proof}

\subsection{Lacunary hyperbolic groups}

A group is said to be \emph{lacunary hyperbolic} if one of its asymptotic cones is an $\R$-tree. We will use the following useful characterisation.

\begin{theorem}[{\cite[Theorem 1.1]{lacunary}}]
\label{lacunary}

    Let $G_0 \to G_1 \to \cdots$ be a directed sequence of epimorphisms of hyperbolic groups, with colimit $G_\infty$. Let $S_0$ be a finite generating set for $G_0$, and let $S_i$ be its pushforward to $G_i$; let $\delta_i$ be the hyperbolicity constant of $\mathrm{Cay}(G_i, S_i)$. For each $i \in \N$, let $\rho_i$ be the radius of the largest ball centred at the identity on which the quotient $G_i \to G_{i+1}$ is injective.

    If $\delta_i = \mathrm{o}(\rho_i)$, then $G_\infty$ is lacunary hyperbolic.
\end{theorem}

As we mentioned in the introduction, many earlier constructions of exotic values of von Neumann dimensions associated to groups involved lamplighters. Hence we record the following:

\begin{proposition}[\cite{lacunary}]
\label{lamplighters}

    Let $H$ be a non-trivial group, then $H \wr \Z$ cannot be a subgroup of a lacunary hyperbolic group.
\end{proposition}

\begin{proof}
    Up to passing to a subgroup, it suffices to show this for $H = \Z$ or $H = \Z/p$ for a prime $p$. When $H = \Z$, this follows from the more general fact that a lacunary hyperbolic group cannot contain a copy of $\Z^2$ by \cite[Corollary 3.21(a)]{lacunary}. For $H = \Z/p$, it follows from \cite[Corollary 3.21(c)]{lacunary} and its proof, see \cite[Remark 3.22]{lacunary}.
\end{proof}

\section{Rationals and hyperbolic groups}

In this section we prove \cref{main hyp} and \cref{pichot}.

\subsection{Realising natural numbers}

We start with the torsion-free case, which is simpler, and sufficient for \cref{pichot}.

\begin{theorem}
\label{pichot precise}

    Let $G$ be a torsion-free cocompact lattice in the isometry group of the octonionic hyperbolic plane, and let $B \subset G$ be a finite subset. Then for every $m \in \N$, there exists a torsion-free hyperbolic quotient $G \to G_m$, injective on $B$, of cohomological dimension $16$, such that for every unitary $G_m$-representation $V$, the second homology is $H_2(G_m; V) \cong V^m$.
\end{theorem}

\begin{proof}
    The group $G$ is torsion-free hyperbolic of cohomological dimension $16$, and by \cref{t2 lattices} it has property $[T_2]$. We construct $G_m$ by induction on $n$, starting with $G_0 = G$, which works by \cref{tn homology}.
    
    Suppose by induction that $G_m$ has been constructed. Apply \cref{induction step prescribed power} with the finite set $B$ (and any $x, H$), let $g$ be the resulting element and $G_{m+1}$ the Dehn filling quotient. Hence $G_{m+1}$ is non-elementary hyperbolic, the quotient $G \to G_{m+1}$ is injective on $B$, and $(G_m, \langle g \rangle, \langle g \rangle)$ is a Cohen--Lyndon triple. Now $G$ is the fundamental group of a closed $16$-dimensional aspherical manifold, so it has cohomological dimension $16$, and $H^{16}(G; \Z) \cong \Z$. The first part of \cref{excision elementary} therefore implies that the cohomological dimension of $G_{m+1}$ is $16$, hence in particular that $G_{m+1}$ is torsion-free. Moreover, since $G_m$ has property $(T)$, applying \cref{tn homology} and the second part of \cref{excision elementary}, we obtain, for every unitary $G_{m+1}$-representation $V$, a short exact sequence
    \[0 \to H_2(G_m; V) \to H_2(G_{m+1}; V) \to V \to 0.\]
    By the induction hypothesis the first term is $V^m$. Therefore $H_2(G_{m+1}; V) \cong V^{m+1}$, which proves the claim and hence the theorem.
\end{proof}

\begin{remark}
    The groups $G_m$ of course fail to have property $(T_2)$, despite being quotients of the $(T_2)$ group $G$, as remarked in \cite{FFF}. Still, the second homology (and similarly the second cohomology) functor enjoys some strong form of rigidity: in a way, its output can be read directly from its input, and the same will be true in the next constructions. It would be interesting to explore whether a (co)homological rigidity of this form characterises quotients of $(T_2)$ groups.
\end{remark}

\begin{proof}[Proof of \cref{pichot}]
    We have $\btt(G) = 0$ by \cref{tn betti quotient}. For each $i \in \N$, we let $B_i$ be the ball of radius $i$ in $G$, for a fixed choice of finite generating set, and $m_i \in \N$ as in the statement of \cref{pichot}. Then \cref{pichot precise} builds a group $G_i$, which by \cref{l2 vs vN} has
    \[\btt(G_i) = \dim_{\vN(G_i)} H_2(G_i; \ell^2(G_i))=\dim_{\vN(G_i)} \ell^2(G_i)^{m_i} = m_i.\]
    Since $G \to G_i$ is injective on larger and larger balls, $G_i \xrightarrow{i \to \infty} G$ in the space of marked groups. The groups $G_i$ have property $(T)$ being quotients of $G$. The other properties of $G_i$ are all contained in the statement of \cref{pichot precise}.
\end{proof}

\begin{remark}
\label{t2 sufficient}

    In \cref{pichot precise}, we used the full power of property $[T_2]$, but the proof shows that if $V$ is assumed to have no non-zero invariant vectors, then property $(T_2)$ suffices. In the proof of \cref{pichot}, we applied \cref{pichot precise} only with $V = \ell^2(G_i)$, which has no non-zero invariant vectors, so property $(T_2)$ would have been sufficient. Similarly, the computation of second $\ell^2$-Betti numbers in the rest of this section and in the next could be obtained only using property $(T_2)$.
\end{remark}

\subsection{Realising rational numbers}

Next we consider arbitrary rational numbers, completing the proof of \cref{main hyp}.

\begin{theorem}[corresponding to~\cref{main hyp}]
\label{main hyp realisation}

Let $x \in \Q_{\geq 0}$ and let $\pi$ be the set of prime factors of the denominator of $x$. Then there exists a hyperbolic group $G$ with property $(T)$ and cohomological dimension $16$ modulo $\pi$, such that $\btt(G) = x$.
\end{theorem}

\begin{proof}
Let $G_0$ be a torsion-free cocompact lattice in the isometry group of the octonionic hyperbolic plane, which has Euler characteristic $\chi \geq 1$ \cite[Corollary 5.16(1) on p.~231]{luck}. If $x = 0$, then we are done. Otherwise let $x = \frac{m}{k}$ be a reduced fraction, where $m, k \in \N$. We will build by induction a sequence of epimorphisms of hyperbolic groups $G_0 \to G_1 \to \cdots \to G_m$ and elements $g_i \in G_i$, such that:
\begin{enumerate}
\item $G_i$ is non-elementary hyperbolic with no non-trivial finite normal subgroups, and cohomological dimension $16$ modulo $\pi$.
\item The element $g_i \in G_i$ has infinite order, and for all $j > i$ its image in $G_j$ has order $k$.
\item\label{hyp item iso homology} Each quotient $G_i \to G_{i+1}$ induces isomorphisms in homology and cohomology with coefficients in $R[G_{i+1}]$-modules and degree at least $3$, where $R$ is any ring where every prime from $\pi$ is invertible.
\item For every unitary $G_i$-representation $V$, the second homology is
\[H_2(G_i; V) = \bigoplus_{j < i} V_{g_j}.\]
\end{enumerate}
The group $G = G_m$ is then non-elementary hyperbolic, with no non-trivial finite normal subgroups and cohomological dimension $16$ modulo $\pi$. It is a quotient of $G_0$ hence it has property $(T)$. Each $g_i : i < m$ has order $k$ in $G$, so by Propositions \ref{l2 vs vN} and \ref{subgroup induction}:
    \begin{align*}
        \dim_{\vN(G)} \ell^2(G)_{g_i} &= \dim_{\vN(G)} H_0(\langle g_i \rangle, \ell^2(G)) = \dim_{\vN(G)} H_0(\langle g_i \rangle, \vN(G)) \\
        &= \dim_{\vN(\langle g_i \rangle)} H_0(\langle g_i \rangle, \vN(\langle g_i \rangle)) = \btz(\langle g_i \rangle) = \frac{1}{k}.
    \end{align*}
    By \cref{l2 vs vN} and additivity we therefore have
    \[\btt(G) = \dim_{\vN(G)} H_2(G; \ell^2(G)) = \sum\limits_{0 \leq i < m} \frac{1}{k} = x.\]
By \cref{hyp item iso homology}, for all $n \geq 3$ we have
    \[\btn(G_m) = \dim_{\vN(G_x)} H_n(G_m; \vN(G_m)) = \dim_{\vN(G_m)} H_n(G_0; \vN(G_m)) = \btn(G_0 \to G_x).\]
    Since $G_m$ has property $(T)$, $\bto(G_m) = 0$ by \cite[Corollary 6]{bekkavalette}. Hence
    \begin{align*}
        \leuler(G_m) &= \sum\limits_{n \in \N} (-1)^n \btn(G_m) = \btt(G_m) + \sum\limits_{n \geq 3} (-1)^n \btn(G_m) \\
        &= \btt(G_m) + \sum\limits_{n \geq 3} (-1)^n \btn(G_0 \to G_m) = \btt(G_m) + \chi,
    \end{align*}
    where the last equality follows from \cref{tn betti quotient}.
    
It remains to construct the sequence. Suppose by induction that $G_i$ has been constructed, for some $0 \leq i < n$. Apply \cref{induction step prescribed power} with $k$ (and any $x, H$), let $g_i$ be the resulting element and $G_{i+1}$ the Dehn filling quotient; we also choose $B$ to be sufficiently large such that for all $j < i$, the element $g_j$ still has order $k$ in $G_{i+1}$. Hence $G_{i+1}$ is non-elementary hyperbolic with no non-trivial finite normal subgroup, and $(G_i, \langle g_i \rangle, \langle g_i^k \rangle)$ is a Cohen--Lyndon triple. Induction and the first part of \cref{excision elementary} imply that $G_{i+1}$ has cohomological dimension $16$ modulo $\pi$. Moreover, the second part and property $(T)$ give, for every unitary $G_{i+1}$-representation $V$, a short exact sequence
    \[0 \to H_2(G_i; V) \to H_2(G_{i+1}; V) \to V_{g_i} \to 0.\]
    This concludes the construction, hence the proof.
\end{proof}

\section{Reals and simple groups}

In this section we prove \cref{main}. Because the proof has many moving parts, we start with a simpler version, that only involves one prime, and does not give simplicity.

\begin{theorem}
\label{main simplified}

Let $p$ be a prime. There exists a family of groups $(G_x)_{x \in \R_{>0}}$ and an integer $\chi \ge 1$ with the following properties.
\begin{enumerate}
    \item $G_x$ is simple, has property $(T)$, and is lacunary hyperbolic.
    \item $G_x$ has cohomological dimension $16$ modulo $p$.
    \item  The second $\ell^2$-Betti number is $\btt(G_x) = x$.
    \item The $\ell^2$-Euler characteristic is $\leuler(G_x)= \btt(G_x)+\chi$.
\end{enumerate}
\end{theorem}

\begin{proof}
    We fix a positive real $x \in \R_{>0}$ and choose a base-$p$ representation
    \[\sum\limits_{i \geq 1} \frac{1}{p^{k_i}},\]
    	for some $k_i \in \N$. To treat all numbers in a unified way, if $x$ is $p$-adic, we still choose a periodic representation.
    The group $G_x$ will arise as the colimit of a direct sequence of epimorphisms of hyperbolic groups $G_0 \to G_1 \to \cdots$, where $G_0$ is a torsion-free cocompact lattice in the isometry group of the octonionic hyperbolic plane, which has Euler characteristic $\chi \geq 1$ \cite[Corollary 5.16(1) on p.~231]{luck}. We fix a finite generating set $S_0$ for $G_0$ and let $S_i$ be the pushforward to $G_i$. We also specify elements $g_i \in G_i$, and will prove the following additional properties.
    \begin{enumerate}
        \item\label{MS item hyp} $G_i$ is non-elementary hyperbolic with no non-trivial finite normal subgroups.
        \item\label{MS item inj} $\mathrm{Cay}(G_i, S_i)$ is $\delta_i$-hyperbolic, and the quotient $G_i \to G_{i+1}$ is injective on the ball of radius $i \delta_i$ centred at the identity.
        \item\label{MS item iso homology} Each quotient $G_i \to G_{i+1}$ induces isomorphisms in homology and cohomology with coefficients in $R[G_{i+1}]$-modules and degree at least $3$, where $R$ is any ring where $p$ is invertible.
        \item\label{MS item elements} The element $g_i \in G_i$ has infinite order, and for all $j > i$ its image in $G_j$ has order $p^{k_i}$.
        \item\label{MS item H2} For every unitary $G_i$-representation $V$, the second homology is
        \[H_2(G_i; V) = \bigoplus_{j < i} V_{g_j}.\]
    \end{enumerate}
    
    Let us first prove that the colimit $G_x$ has the desired properties. The group $G_0$ has property $(T)$ by \cref{t2 lattices}, hence so does $G_x$, being its quotient. \cref{MS item inj} implies that $G_x$ is lacunary hyperbolic, by \cref{lacunary}. \cref{MS item iso homology} and \cref{dimensions} imply that $G_x$ has cohomological dimension $16$ modulo $p$.

    Next, we compute $\btt(G_x)$. By \cref{colimit homology}, we have
    \[H_2(G_x; \ell^2(G_x)) \cong \bigoplus_{i \geq 1} \ell^2(G_x)_{g_i}.\]
    Each $g_i$ has order $p^{k_i}$ in $G_x$, hence by Propositions \ref{l2 vs vN} and \ref{subgroup induction}:
    \begin{align*}
        \dim_{\vN(G_x)} \ell^2(G_x)_{g_i} &= \dim_{\vN(G_x)} H_0(\langle g_i \rangle, \ell^2(G_x)) = \dim_{\vN(G_x)} H_0(\langle g_i \rangle, \vN(G_x)) \\
        &= \dim_{\vN(\langle g_i \rangle)} H_0(\langle g_i \rangle, \vN(\langle g_i \rangle)) = \btz(\langle g_i \rangle) = p^{-k_i}.
    \end{align*}
    By \cref{l2 vs vN} and \cref{direct sums}, we therefore have
    \[\btt(G_x) = \dim_{\vN(G_x)} H_2(G_x; \ell^2(G_x)) = \sum\limits_{i \geq 1} \frac{1}{p^{k_i}} = x.\]
    
    Finally, we compute $\leuler(G_x)$. By \cref{MS item iso homology} and \cref{colimit homology}, for all $n \geq 3$ we have
    \[\btn(G_x) = \dim_{\vN(G_x)} H_n(G_x; \vN(G_x)) = \dim_{\vN(G_x)} H_n(G_0; \vN(G_x)) = \btn(G_0 \to G_x).\]
    Since $G_x$ has property $(T)$, $\bto(G_x) = 0$ by \cite[Corollary 6]{bekkavalette}. Hence
    \begin{align*}
        \leuler(G_x) &= \sum\limits_{n \in \N} (-1)^n \btn(G_x) = \btt(G_x) + \sum\limits_{n \geq 3} (-1)^n \btn(G_x) \\
        &= \btt(G_x) + \sum\limits_{n \geq 3} (-1)^n \btn(G_0 \to G_x) = \btt(G_x) + \chi,
    \end{align*}
    where the last equality follows from \cref{tn betti quotient}.
    
    \medskip

    It remains to construct the directed sequence $G_0 \to G_1 \to \cdots$ and the elements $g_i \in G_i$ so that the items above are satisfied. Starting with $G_0$, note that \cref{MS item H2} holds by \cref{tn homology}, because $G_0$ has property $[T_2]$, and there is nothing to check for the other items, before $g_0$ has been chosen. So suppose by induction that $G_0, \ldots, G_i$ and the elements $g_0, \ldots, g_{i-1}$ have been constructed so that all items above are satisfied. Apply \cref{induction step prescribed power} with $k = k_i$, a finite set $B$ to be specified later (and any $x, H$). Let $g_i$ be the resulting element and $G_{i+1}$ the Dehn filling quotient. \cref{induction step prescribed power} directly gives \cref{MS item hyp}. By choosing $B$ large enough, we can ensure that \cref{MS item inj} is satisfied, and that the elements $g_j : j < i$ maintain the same order in $G_{i+1}$; this together with \cref{induction step prescribed power} gives \cref{MS item elements}.
    
    Because $(G_i, \langle g_i \rangle, \langle g_i^k \rangle)$ is Cohen--Lyndon, again by \cref{induction step prescribed power}, we can apply \cref{excision elementary}, the first part gives \cref{MS item iso homology}. The second part, together with the fact that $G_i$ has property $(T)$, gives a short exact sequence
    \[0 \to H_2(G_i; V) \to H_2(G_{i+1}; V) \to V_{g_i} \to 0.\]
    The induction hypothesis gives
    \[H_2(G_{i+1}; V) = \bigoplus\limits_{j \leq i} V_{g_j},\]
    proving \cref{MS item H2}. This concludes the construction, hence the proof.
\end{proof}

We now prove \cref{main}. The proof will be similar, but more technical, in that we need to leverage an additional prime, and the full power of \cref{induction step prescribed power}, to make the group $G_x$ simple. This improvement employs a trick from \cite{osinthom}, where it was used to construct simple groups with large first $\ell^2$-Betti numbers. The following will be our starting point, this is also the only place where we use \cref{induction step large power} for the induction step.

\begin{lemma}
\label{main starting point}

    Let $G$ be a torsion-free cocompact lattice in the isometry group of the octonionic hyperbolic plane. Let $q$ be a prime and let $\varepsilon > 0$. Then there exists a quotient $G \to G_0$ with the following properties.
    \begin{enumerate}
        \item $G_0$ is non-elementary hyperbolic with no non-trivial finite normal subgroups.
        \item The quotient $G \to G_0$ induces isomorphisms in homology and cohomology with coefficients in $R[G_0]$-modules and degree at least $3$, where $R$ is any ring where $q$ is invertible.
        \item $G_0$ is generated by elements $\{ a_1, \ldots, a_m \}$, each of which has order a power of $q$.
        \item There is a finite set $\mathcal{E}$ of finite subgroups of $G_0$, such that for every unitary $G_0$-representation $V$, the second homology is
        \[H_2(G_0; V) = \bigoplus_{E \in \mathcal{E}} V_E,\]
        and moreover
        \[ \sum\limits_{E \in \mathcal{E}} |E|^{-1} < \varepsilon.\]
    \end{enumerate}
\end{lemma}

\begin{proof}
    We start by choosing a generating set $\{a_1^0, \ldots, a_m^0\}$ of $G$. For $i = 1, \ldots, m$, we construct by induction a quotient $G \to G^i$ with the following properties.
    \begin{enumerate}
        \item\label{SP item hyp} $G^i$ is non-elementary hyperbolic with no non-trivial finite normal subgroups.
        \item\label{SP item iso} The quotient $G \to G^i$ induces isomorphisms in homology and cohomology with coefficients in $R[G^i]$-modules and degree at least $3$, where $R$ is any ring where $q$ is invertible. In particular $G^i$ has cohomological dimension $16$ modulo $q$, and so it only contains $q$-torsion \cite[Proposition 4.11 on p.~63]{bieri}.
        \item\label{SP item gen} $G^i$ is generated by elements $\{a_1^i, \ldots, a_m^i\}$, where $a_1^i, \ldots, a_i^i$ have order a power of $q$.
        \item\label{SP item H2} There is a finite set of finite subgroups $\mathcal{E}_i$ of $G^i$, such that for every unitary $G^i$-representation $V$, the second homology is
        \[H_2(G^i; V) = \bigoplus_{E \in \mathcal{E}_i} V_E,\]
        and moreover
        \[ \sum\limits_{E \in \mathcal{E}_i} |E|^{-1} < \varepsilon.\]
    \end{enumerate}
    Then the final group $G^m$ will be the group $G_0$ in the statement of the lemma, with $\mathcal{E} = \mathcal{E}_m$ and $a_j = a_j^m$. The base case for $i = 0$ is $G$ itself, with $\mathcal{E} = \emptyset$, where we do not need to check anything for the first three items, and \cref{SP item H2} follows from \cref{tn homology} and \cref{t2 lattices}.

    Suppose by induction that $G^i$ has been constructed. If $a_{i+1}^i$ is torsion, then it is $q$-torsion, so we set $G^{i+1} = G^i$, $a^{i+1}_j = a^i_j$ for all $j$, and $\mathcal{E}_{i+1} = \mathcal{E}_i$. Otherwise, consider the infinite virtually cyclic group $E_{G^i}(a_{i+1})$. This has a unique maximal finite normal subgroup of order a power of $q$, and the quotient is isomorphic to $\Z$ (or possibly to $D_\infty$, in case $q = 2$). We replace $a_{i+1}^i$ with a root $a^{i+1}_{i+1}$ that generates the $\Z$ quotient (or the index-$2$ subgroup of the $D_\infty$ quotient), so that $[E_{G^i}(a^{i+1}_{i+1}) : \langle a^{i+1}_{i+1} \rangle]$ is a power of $q$. We now let $G^{i+1}$ be the group-theoretic Dehn filling for the triple $(G^i, E_{G^i}(a^{i+1}_{i+1}), \langle (a^{i+1}_{i+1})^k \rangle)$, where $k$ is a sufficiently large power of $q$. By \cref{induction step large power}, we see that $G^{i+1}$ satisfies Items \ref{SP item hyp} and \ref{SP item gen}, choosing $a^{i+1}_j = a^i_j$ for all $j \neq i+1$. We also choose $k$ to be large enough so that the orders of $a_1^i, \ldots, a_i^i$ and their images $a_1^{i+1}, \ldots, a_i^{i+1}$ have the same order, and each finite subgroup from $\mathcal{E}_i$ embeds into $G^{i+1}$. Moreover, by \cref{induction step large power} again, this triple is Cohen--Lyndon, so \cref{excision elementary} applies, and proves \cref{SP item iso}. Finally, let $E \coloneqq E_{G^i}(a_{i+1}^{i+1}) / \langle (a_{i+1}^{i+1})^k \rangle$. \cref{excision elementary} and property $(T)$ for $G^i$ give a short exact sequence
    \[0 \to H_2(G^i; V) \to H_2(G^{i+1}; V) \to V_E \to 0,\]
    so setting $\mathcal{E}_{i+1} = \mathcal{E}_i \cup \{ E \}$, we obtain the first part of \cref{SP item H2}. The second part follows again by induction, if we choose $k$ to be sufficiently large, and we conclude.
\end{proof}

\begin{proof}[Proof of \cref{main}]
    We proceed as in the proof of \cref{main simplified}, using the prime $p \neq q$, but starting with the group $G_0$ from \cref{main starting point}. The result will be a group $G_x$ with the same properties as in \cref{main simplified}, with the exception that cohomological dimension now has to be taken modulo $\{p, q\}$, and $\btt(G_x)$ is computed starting from
    \[\btt(G_0 \to G_x) = \dim_{\vN(G_x)} \Bigl(\bigoplus_{E \in \mathcal{E}} \vN(G_\alpha)_E\Bigr) = \sum\limits_{E \in \mathcal{E}} |E|^{-1}.\]
    We let this value be $\varepsilon < x$. The same computations go through, where now work with a base-$p$ representation of $x - \varepsilon$ instead. For the finiteness of the higher $\ell^2$-Betti numbers of~$G_x$ we refer to~\cref{singer remark}. 

    In each induction step in the proof of \cref{main simplified}, we chose special elements $g_i \in G_i$, using \cref{induction step prescribed power} with arbitrary $x, H$. We will now explain how to use the full power of \cref{induction step prescribed power} in the construction so that this time the colimit is simple. Enumerate the set $\mathcal{A} \coloneqq \{a_1, \ldots, a_m \} \times G_0$. When the group $G_i$ has been built, let $(a, h)$ denote the $i$-th element of $\mathcal{A}$. If $h$ is trivial, then we pick $g_i$ to be an arbitrary special element, and proceed as before. Otherwise, we use \cref{induction step prescribed power} to pick $g_i \in a^{-1} \normal{h}_{G_i}$, which is possible since $\normal{h}_{G_i}$ is suitable by \cref{normal suitable}. We claim that now $a \in \normal{h}_{G_{i+1}} \leq G_{i+1}$. Indeed, consider the image of $a$ in the quotient $G_{i+1}/\normal{h}_{G_{i+1}}$. On the one hand, it has order a power of $q$, because $a$ has order a power of $q$ in $G_0$. On the other hand, it equals the image of $g_i$, which has order a power of $p$ in $G_{i+1}$, by construction of the quotient $G_i \to G_{i+1}$. Because $p \neq q$, we see that $a$ must be trivial in $G_{i+1}/\normal{h}_{G_{i+1}}$, which proves our claim. 

    As a result, for every $(a, h) \in \mathcal{A}$, either $h$ is trivial in $G_x$, or $a$ lies in the normal closure of $h$ in $G_x$. Because $a$ ranges over a generating set for $G_0$ and $h$ ranges over $G_0$, this shows that every non-trivial element of $G_x$ normally generates, that is, $G_x$ is simple.
\end{proof}

\begin{remark}\label{singer remark}
    Note that the constructions give some information about higher $\ell^2$-Betti numbers as well. Indeed, $G_x$ is obtained from a torsion-free cocompact lattice in the isometry group of the octonionic hyperbolic plane, denoted $G$, by a sequence of group-theoretic Dehn fillings satisfying \cref{excision elementary}. Using this and \cref{colimit homology} we have
    \[\btn(G_x) = \btn(G \to G_x)\]
    for all $n \geq 3$. Since $G$ is of type $F$, this implies that all $\ell^2$-Betti numbers of $G_x$ are finite. We know from \cref{t2 lattices} that $G$ has property $(T_3)$, so \cref{tn betti quotient} gives $\bttt(G_x) = 0$.

    Moreover, $G$ satisfies the Singer conjecture \cite[Corollary 5.16 on p.~231]{luck}, that is $\btn(G) = 0$ for all $n \neq 8$ and $\betti_8(G) = \chi \geq 1$. This does not give information about $\btn(G \to G_x)$, however, assuming L{\"u}ck's approximation conjecture \cite[Chapter~13]{luck}, and repeating the construction ensuring that $G \to G_x$ is injective on a large ball centred at the identity, we would have that $\btn(G\to G_x)$ is close to $\btn(G)$. It would then follow that $\btn(G_x)$ is close to $0$, for $n \geq 4, n \neq 8$, and close to $\chi$ for $n = 8$.
\end{remark}

\begin{remark}
    Taking $V = \C$ to be the trivial unitary representation, the proofs of Theorems~\ref{main simplified} and~\ref{main}, and \cref{colimit homology}, show that $H_2(G_x; \C)$ is infinite-dimensional. It follows that the groups $G_\alpha$ do not have type $FP_2(\Q)$.
\end{remark}

\section{An elementary construction}

We conclude with the elementary construction giving \cref{suspension}. We start with two basic product formulas.

\begin{proposition}[{\cite[Theorem 6.54(5) on p.~266]{luck}}]
\label{direct product}

    Let $G_1, G_2$ be two groups. Then
    \[\btn(G_1 \times G_2) = \sum\limits_{i+j = n} \bti(G_1) \btj(G_2).\]
\end{proposition}

\begin{proposition}
\label{free product}

    Let $(G_i)_{i \in \N}$ be infinite groups, and let $G$ be their free product. Then $\bto(G) = \infty$ and
    \[\btn(G) = \sum\limits_{i \in \N} \btn(G_i)\]
    for all $n \geq 2$.
\end{proposition}

This is usually stated for finite free products only \cite[Theorem 1.35(5) on p.~38]{luck} so we include a proof.

\begin{proof}
    Let $P_i$ denote the free product $G_1 \ast \cdots \ast G_i$, so $P_{i+1} = P_i \ast G_{i+1}$. By Mayer--Vietoris, for all $n \geq 2$ we have an isomorphism
    \[H_n(P_{i+1}; \vN(G)) \cong H_n(P_i; \vN(G)) \oplus H_n(G_{i+1}; \vN(G)),\]
    which shows by induction that
    \[H_n(P_i; \vN(G)) = \bigoplus_{j \leq i} H_n(G_j; \vN(G)).\]
    By Mayer--Vietoris again, we have an exact sequence
    \[0 \to H_1(P_{i+1}; \vN(G)) \to H_1(P_i; \vN(G)) \oplus H_1(G_{i+1}; \vN(G)) \to \vN(G) \to \vN(G)_{P_{i+1}}.\]
    By \cref{subgroup induction}, the last term has dimension
    \[\dim_{\vN(G)} H_0(P_{i+1}; \vN(G)) = \dim_{\vN(P_{i+1})} H_0(P_{i+1}; \vN(P_{i+1})) = \btz(P_{i+1}) = 0.\]
    So by induction
    \[H_1(P_i; \vN(G)) = \left( \bigoplus_{j \leq i} H_1(G_j; \vN(G)) \right) \oplus V_i,\]
    where $V_i$ is a submodule of $\vN(G)^{i-1}$ of codimension $0$. We conclude from Propositions \ref{colimit homology} and \ref{subgroup induction}, together with the fact that $\dim_{\vN(G)}$ respects directed unions \cite[Theorem 0.6(b)]{lueck-dimension}.
\end{proof}

We can now begin the construction.

\begin{proposition}
\label{countable realisation}

    Let $p$ be a prime. There exists a family of groups $(C_x)_{x \in \R_{> 0}}$ with the following properties.
    \begin{enumerate}
        \item Each $C_x$ is countable and residually finite.
        \item $\bto(C_x) = \infty, \btt(C_x) = x$.
        \item Each $C_x$ has cohomological dimension $2$ modulo $p$. In particular, $\btn(C_x) = 0$ for all $n \geq 3$.
    \end{enumerate}
\end{proposition}

\begin{proof}
    Write $x$ in base-$p$ as
    \[x = \sum_{i \in \N} \frac{r_i}{p^i},\]
    where $r_0 \in \N$; $r_i \in \{0, 1, \ldots, p-1\}$ for $i \geq 1$. For elements in $\Z\left[\frac{1}{p}\right]$ we pick the periodic representation, which allows to have infinite sums for every $x$. Now for each $i$ such that $r_i \neq 0$ we let $A_i \coloneqq \Z/p^i\Z \times F_{r_i + 1}^2$, where $F_r$ is the free group of rank $r$. Then
    \[\btz(A_i) = \bto(A_i) = 0; \btt(A_i) = \frac{r_i}{p^i},\]
    by \cref{direct product}. Moreover, $A_i$ has cohomological dimension $2$ modulo $p$.

    Now let
    \[C_x \coloneqq \Asterisk_{i \geq 0} A_i.\]
    By \cref{free product}, we have the desired values of $\ell^2$-Betti numbers. Moreover, $C_x$ has cohomological dimension $2$ modulo $p$, by \cref{dimensions}. Finally, $C_x$ is residually finite, being a free product of residually finite groups.
\end{proof}

\begin{proof}[Proof of \cref{suspension}]
    Let $C_x$ be the groups from \cref{countable realisation}. We embed each $C_x$ into a finitely generated residually finite group $R_x$ following \cite{WZ}. The group $R_x$ is an HNN extension of $F_2 \times C_x$ over a free subgroup, therefore it has cohomological dimension $3$ modulo~$p$. Moreover, $C_x$ is separable, that is closed in the profinite topology \cite[Lemma 4]{WZ}.
    
    Consider the group $R_x \times F_2^3$: it has cohomological dimension $6$ modulo $p$, and contains $C_x$ as a separable subgroup. Moreover, $\btn(R_x \times F_2^3) = 0$ for $n \leq 3$, by \cref{direct product}. Define
    \[G_x \coloneqq (R_x \times F_2^3) \Asterisk_{C_x} (R_x \times F_2^3),\]
    where the double is over the copy of $C_x$ coming from the first coordinate of each factor. Because $C_x$ is separable, $G_x$ is residually finite, and moreover it has cohomological dimension $6$ modulo $p$. By Mayer--Vietoris:
    \[H_n((R_x \times F_2^3); \vN(G_x))^2 \to H_n(G_x; \vN(G_x)) \to H_{n-1}(C_x; \vN(G_x)) \to H_{n-1}(R_x \times F_2^3); \vN(G_x))^2.\]
    Taking $\dim_{\vN(G_x)}$ and $n \leq 3$, the first and last term give $2 \btn(R_x \times F_2^3) = 2\btnmo(R_x \times F_2^3) = 0$, by \cref{subgroup induction}, so
    \[\btn(G_x) = \btnmo(C_x).\]
    The value of $\ell^2$-Betti numbers are as desired, by \cref{countable realisation}.
\end{proof}

\begin{remark}
    The construction is quite flexible, and can be made more elementary, or less elementary, depending on what properties one wants to impose, if we forgo residual finiteness.
    \begin{enumerate}
        \item We can choose $R_x$ to be a finitely generated group that is an HNN extension of $F_2 \ast C_x$ over a free subgroup, therefore has cohomological dimension $2$ modulo $p$ \cite{HNN}. This is the first and easiest embedding result for countable groups into finitely generated groups. Then $R_x \times F_2^3$ has cohomological dimension $5$ modulo $p$, hence so does $G_x$.
        \item We can bring the dimension down by one more, up to possibly adding some primes in the torsion. Let $L$ be a generalised triangle group that is hyperbolic and has property $(T)$ \cite{triangle}. Then there exists a finite set of primes $\pi$ (the ones appearing in the orders of the finite vertex groups) such that $L$ has cohomological dimension $2$ modulo $\pi$. Using \cite[Theorem 4.1]{dimensions}, we can embed $C_x$ into a group $R_x$ that is a quotient of $L$ and has cohomological dimension $2$ modulo $\pi \cup \{p\}$. In particular $R_x$ has property $(T)$ and so $\bto(R_x) = 0$ by \cite[Corollary 6]{bekkavalette}. So in the construction above it suffices to take $R_x \times F_2^2$ to get the correct vanishing of $\ell^2$-Betti numbers, and thus the resulting $G_x$ has cohomological dimension $4$ modulo $\pi \cup \{p\}$.
        \item If we forgo the finiteness of the dimension, then we can gain control on all other $\ell^2$-Betti numbers. We first choose $R_x'$ to be any finitely generated group containing $C_x$, and then let $R_x \coloneqq R_x' \wr \Z$, which has vanishing $\ell^2$-Betti numbers in all degrees \cite[Remark 7.6]{WWZZ}. So in the construction above we do not need to take products with $F_2$, and directly get $\btn(G_x) = \btnmo(C_x)$ for all $n \in \N$. Therefore $\btn(G_x) = 0$ for all $n \geq 4$.
    \end{enumerate}
\end{remark}

\begin{corollary}
\label{ME easy}
    Let $(G_x)_{x \in \R_{>0}}$ be as in \cref{suspension}. Then the groups $G_x \ast \Z$ are pairwise non-measure equivalent.
\end{corollary}

\begin{proof}
    If $G_x \ast \Z$ and $G_y \ast \Z$ are measure equivalent with index $c>0$, then $\btn(G_x \ast \Z) = c\btn(G_y \ast \Z)$ for all $n \in \N$ by Gaboriau's theorem \cite{gaboriau}. So $1 = \bto(G_x \ast \Z) = c \bto(G_y \ast \Z) = c$, that is $c = 1$, and so $x = \bttt(G_x \ast \Z) = c \bttt(G_y \ast \Z) = cy = y$.
\end{proof}

\footnotesize

\bibliographystyle{amsalpha}
\bibliography{ref}

\providecommand{\bysame}{\leavevmode\hbox to3em{\hrulefill}\thinspace}
\providecommand{\MR}{\relax\ifhmode\unskip\space\fi MR }
\providecommand{\MRhref}[2]{%
  \href{http://www.ams.org/mathscinet-getitem?mr=#1}{#2}
}
\providecommand{\href}[2]{#2}
\begin{thebibliography}{WWZZ25}

\bibitem[AH21]{abbotthull}
C.~Abbott and M.~Hull, \emph{Random walks and quasi-convexity in acylindrically
  hyperbolic groups}, J. Topol. \textbf{14} (2021), no.~3, 992--1026.
  \MR{4503955}

\bibitem[AMO07]{SQ}
G.~Arzhantseva, A.~Minasyan, and D.~Osin, \emph{The {SQ}-universality and
  residual properties of relatively hyperbolic groups}, J. Algebra \textbf{315}
  (2007), no.~1, 165--177. \MR{2344339}

\bibitem[Ati76]{atiyah}
M.~F. Atiyah, \emph{Elliptic operators, discrete groups and von {N}eumann
  algebras}, Colloque ``{A}nalyse et {T}opologie'' en l'{H}onneur de {H}enri
  {C}artan ({O}rsay, 1974), Ast\'{e}risque, No. 32-33, Soc. Math. France,
  Paris, 1976, pp.~43--72. \MR{420729}

\bibitem[Aus13]{austin}
T.~Austin, \emph{Rational group ring elements with kernels having irrational
  dimension}, Proc. Lond. Math. Soc. (3) \textbf{107} (2013), no.~6,
  1424--1448. \MR{3149852}

\bibitem[BdlHV08]{BHV}
B.~Bekka, P.~de~la Harpe, and A.~Valette, \emph{Kazhdan's property ({T})}, New
  Mathematical Monographs, vol.~11, Cambridge University Press, Cambridge,
  2008. \MR{2415834}

\bibitem[BH99]{bridsonhaefliger}
M.~R. Bridson and A.~Haefliger, \emph{Metric spaces of non-positive curvature},
  Grundlehren der mathematischen Wissenschaften [Fundamental Principles of
  Mathematical Sciences], vol. 319, Springer-Verlag, Berlin, 1999. \MR{1744486}

\bibitem[BHC62]{BHC}
A.~Borel and Harish-Chandra, \emph{Arithmetic subgroups of algebraic groups},
  Ann. of Math. (2) \textbf{75} (1962), 485--535. \MR{147566}

\bibitem[Bie81]{bieri}
R.~Bieri, \emph{Homological dimension of discrete groups}, second ed., Queen
  Mary College Mathematics Notes, Queen Mary College, Department of Pure
  Mathematics, London, 1981. \MR{715779}

\bibitem[Bow12]{bowditch}
B.~H. Bowditch, \emph{Relatively hyperbolic groups}, Internat. J. Algebra
  Comput. \textbf{22} (2012), no.~3, 1250016, 66. \MR{2922380}

\bibitem[Bro94]{brown}
K.~S. Brown, \emph{Cohomology of groups}, Graduate Texts in Mathematics,
  vol.~87, Springer-Verlag, New York, 1994, Corrected reprint of the 1982
  original. \MR{1324339}

\bibitem[BS23]{T2}
U.~Bader and R.~Sauer, \emph{Higher {K}azhdan property and unitary cohomology
  of arithmetic groups}, arXiv preprint arXiv:2308.06517, 2023.

\bibitem[BS25]{ICM}
\bysame, \emph{Higher property {T} and below-rank phenomena of lattices}, arXiv
  preprint arXiv:2511.20192, 2025.

\bibitem[BV97]{bekkavalette}
M.~E.~B. Bekka and A.~Valette, \emph{Group cohomology, harmonic functions and
  the first {$L^2$}-{B}etti number}, Potential Anal. \textbf{6} (1997), no.~4,
  313--326. \MR{1452785}

\bibitem[CCKW22]{triangle}
P.-E. Caprace, M.~Conder, M.~Kaluba, and S.~Witzel, \emph{Hyperbolic
  generalized triangle groups, property ({T}) and finite simple quotients}, J.
  Lond. Math. Soc. (2) \textbf{106} (2022), no.~4, 3577--3637. \MR{4524205}

\bibitem[CIOS23]{wreathlike}
I.~Chifan, A.~Ioana, D.~Osin, and B.~Sun, \emph{Wreath-like products of groups
  and their von {N}eumann algebras {I}: {$\rm W^\ast $}-superrigidity}, Ann. of
  Math. (2) \textbf{198} (2023), no.~3, 1261--1303. \MR{4660139}

\bibitem[CIOS26]{wreathlike2}
\bysame, \emph{Wreath-like products of groups and their von {N}eumann algebras
  {II}: outer automorphisms}, Duke Math. J. \textbf{175} (2026), no.~2,
  287--359. \MR{5030466}

\bibitem[CL63]{CL}
D.~E. Cohen and R.~C. Lyndon, \emph{Free bases for normal subgroups of free
  groups}, Trans. Amer. Math. Soc. \textbf{108} (1963), 526--537. \MR{170930}

\bibitem[DGO17]{DGO}
F.~Dahmani, V.~Guirardel, and D.~Osin, \emph{Hyperbolically embedded subgroups
  and rotating families in groups acting on hyperbolic spaces}, Mem. Amer.
  Math. Soc. \textbf{245} (2017), no.~1156, v+152. \MR{3589159}

\bibitem[FF25]{FFF}
F.~Fournier-Facio, \emph{Stability, approximable quotients, and higher property
  {(T)}}, arXiv preprint arXiv:2512.09180, 2025.

\bibitem[FFS25]{dimensions}
F.~Fournier-Facio and B.~Sun, \emph{Dimensions of finitely generated simple
  groups and their subgroups}, arXiv preprint arXiv:2503.01987, 2025.

\bibitem[Gab00]{gaboriau:cost}
D.~Gaboriau, \emph{Co\^{u}t des relations d'\'{e}quivalence et des groupes},
  Invent. Math. \textbf{139} (2000), no.~1, 41--98. \MR{1728876}

\bibitem[Gab02]{gaboriau}
\bysame, \emph{Invariants {$l^2$} de relations d'\'equivalence et de groupes},
  Publ. Math. Inst. Hautes \'Etudes Sci. (2002), no.~95, 93--150. \MR{1953191}

\bibitem[GMS19]{CL:RH}
D.~Groves, J.~F. Manning, and A.~Sisto, \emph{Boundaries of {D}ehn fillings},
  Geom. Topol. \textbf{23} (2019), no.~6, 2929--3002. \MR{4039183}

\bibitem[Gra14]{grabowski1}
\L. Grabowski, \emph{On {T}uring dynamical systems and the {A}tiyah problem},
  Invent. Math. \textbf{198} (2014), no.~1, 27--69. \MR{3260857}

\bibitem[Gra16]{grabowski2}
\bysame, \emph{Irrational {$l^2$} invariants arising from the lamplighter
  group}, Groups Geom. Dyn. \textbf{10} (2016), no.~2, 795--817. \MR{3513118}

\bibitem[Gro87]{gromov:hyp}
M.~Gromov, \emph{Hyperbolic groups}, Essays in group theory, Math. Sci. Res.
  Inst. Publ., vol.~8, Springer, New York, 1987, pp.~75--263. \MR{919829}

\bibitem[Hat02]{hatcher}
A.~Hatcher, \emph{Algebraic topology}, Cambridge University Press, Cambridge,
  2002. \MR{1867354}

\bibitem[HNN49]{HNN}
G.~Higman, B.~H. Neumann, and N.~Neumann, \emph{Embedding theorems for groups},
  J. London Math. Soc. \textbf{24} (1949), 247--254. \MR{32641}

\bibitem[Hul16]{hull}
M.~Hull, \emph{Small cancellation in acylindrically hyperbolic groups}, Groups
  Geom. Dyn. \textbf{10} (2016), no.~4, 1077--1119. \MR{3605028}

\bibitem[ITD25]{ITD}
A.~Ioana and R.~Tucker-Drob, \emph{A continuum of non-measure equivalent
  groups}, arXiv preprint arXiv:2512.04531, 2025.

\bibitem[KPV15]{kyed+petersen+vaes}
D.~Kyed, H.~D. Petersen, and S.~Vaes, \emph{{$L^2$}-{B}etti numbers of locally
  compact groups and their cross section equivalence relations}, Trans. Amer.
  Math. Soc. \textbf{367} (2015), no.~7, 4917--4956. \MR{3335405}

\bibitem[Kum80]{kumaresan}
S.~Kumaresan, \emph{On the canonical {$k$}-types in the irreducible unitary
  {$g$}-modules with nonzero relative cohomology}, Invent. Math. \textbf{59}
  (1980), no.~1, 1--11. \MR{575078}

\bibitem[LN23]{antonio}
A.~L{\'{o}}pez~Neumann, \emph{Finitely presented simple groups and measure
  equivalence}, Colloq. Math. \textbf{172} (2023), no.~2, 261--279.
  \MR{4589106}

\bibitem[LNP25]{LNP}
A.~L{\'o}pez~Neumann and H.~Paucar, \emph{Quantitative polynomial cohomology
  and applications to ${L}^p$-measure equivalence}, arXiv preprint
  arXiv:2512.18463, 2025.

\bibitem[LS01]{lyndon:schupp}
R.~C. Lyndon and P.~E. Schupp, \emph{Combinatorial group theory}, 1977 ed.,
  Classics in Mathematics, Springer-Verlag, Berlin, 2001. \MR{1812024}

\bibitem[L{\"u}c98]{lueck-dimension}
W.~L{\"u}ck, \emph{Dimension theory of arbitrary modules over finite von
  {N}eumann algebras and {$L^2$}-{B}etti numbers. {I}. {F}oundations}, J. Reine
  Angew. Math. \textbf{495} (1998), 135--162. \MR{1603853}

\bibitem[L{\"u}c02]{luck}
\bysame, \emph{{$L^2$}-invariants: theory and applications to geometry and
  {$K$}-theory}, Ergebnisse der Mathematik und ihrer Grenzgebiete. 3. Folge. A
  Series of Modern Surveys in Mathematics [Results in Mathematics and Related
  Areas. 3rd Series. A Series of Modern Surveys in Mathematics], vol.~44,
  Springer-Verlag, Berlin, 2002. \MR{1926649}

\bibitem[Mat62]{matsushima}
Y.~Matsushima, \emph{On {B}etti numbers of compact, locally symmetric
  {R}iemannian manifolds}, Osaka Math. J. \textbf{14} (1962), 1--20.
  \MR{141138}

\bibitem[New68]{newman}
B.~B. Newman, \emph{Some results on one-relator groups}, Bull. Amer. Math. Soc.
  \textbf{74} (1968), 568--571. \MR{222152}

\bibitem[Ols93]{olshanskii}
A.~Yu. Olshanskii, \emph{On residualing homomorphisms and {$G$}-subgroups of
  hyperbolic groups}, Internat. J. Algebra Comput. \textbf{3} (1993), no.~4,
  365--409. \MR{1250244}

\bibitem[OOS09]{lacunary}
A.~Yu. Olshanskii, D.~Osin, and M.~V. Sapir, \emph{Lacunary hyperbolic groups},
  Geom. Topol. \textbf{13} (2009), no.~4, 2051--2140, With an appendix by M.
  Kapovich and B. Kleiner. \MR{2507115}

\bibitem[Osi07]{osin}
D.~Osin, \emph{Peripheral fillings of relatively hyperbolic groups}, Invent.
  Math. \textbf{167} (2007), no.~2, 295--326. \MR{2270456}

\bibitem[Osi10]{osin:sc}
\bysame, \emph{Small cancellations over relatively hyperbolic groups and
  embedding theorems}, Ann. of Math. (2) \textbf{172} (2010), no.~1, 1--39.
  \MR{2680416}

\bibitem[Osi25]{osin:measures}
\bysame, \emph{{${\rm Out}(F_n)$}-invariant probability measures on the space
  of {$n$}-generated marked groups}, Groups Geom. Dyn. \textbf{19} (2025),
  no.~2, 431--444. \MR{4940646}

\bibitem[OT13]{osinthom}
D.~Osin and A.~Thom, \emph{Normal generation and {$\ell^2$}-{B}etti numbers of
  groups}, Math. Ann. \textbf{355} (2013), no.~4, 1331--1347. \MR{3037017}

\bibitem[Pic06]{pichot}
M.~Pichot, \emph{Semi-continuity of the first {$l^2$}-{B}etti number on the
  space of finitely generated groups}, Comment. Math. Helv. \textbf{81} (2006),
  no.~3, 643--652. \MR{2250857}

\bibitem[PS24a]{PS2}
N.~Petrosyan and B.~Sun, \emph{Cohomology of group theoretic {D}ehn fillings
  {II}}, Adv. Math. \textbf{437} (2024), Paper No. 109412, 56. \MR{4669335}

\bibitem[PS24b]{PS3}
\bysame, \emph{{$L^2$}-{B}etti numbers of {D}ehn fillings}, arXiv preprint
  arXiv:2412.16090, 2024.

\bibitem[PT11]{peterson+thom}
J.~Peterson and A.~Thom, \emph{Group cocycles and the ring of affiliated
  operators}, Invent. Math. \textbf{185} (2011), no.~3, 561--592. \MR{2827095}

\bibitem[Sau05]{sauer-phd}
R.~Sauer, \emph{{$L^2$}-{B}etti numbers of discrete measured groupoids},
  Internat. J. Algebra Comput. \textbf{15} (2005), no.~5-6, 1169--1188.
  \MR{2197826}

\bibitem[Sel60]{selberg}
A.~Selberg, \emph{On discontinuous groups in higher-dimensional symmetric
  spaces}, Contributions to function theory ({I}nternat. {C}olloq. {F}unction
  {T}heory, {B}ombay, 1960), Tata Inst. Fund. Res., Bombay, 1960, pp.~147--164.
  \MR{130324}

\bibitem[Sun20]{PS1}
B.~Sun, \emph{Cohomology of group theoretic {D}ehn fillings {I}:
  {C}ohen-{L}yndon type theorems}, J. Algebra \textbf{542} (2020), 277--307.
  \MR{4019787}

\bibitem[Wei94]{weibel}
C.~A. Weibel, \emph{An introduction to homological algebra}, Cambridge Studies
  in Advanced Mathematics, vol.~38, Cambridge University Press, Cambridge,
  1994. \MR{1269324}

\bibitem[WWZZ25]{WWZZ}
F.~Wu, X.~Wu, M.~Zhao, and Z.~Zhou, \emph{Embedding groups into boundedly
  acyclic groups}, J. Lond. Math. Soc. (2) \textbf{111} (2025), no.~5, Paper
  No. e70164, 39. \MR{4899448}

\bibitem[WZ96]{WZ}
J.~S. Wilson and P.~A. Zalesskii, \emph{An embedding theorem for certain
  residually finite groups}, Arch. Math. (Basel) \textbf{67} (1996), no.~3,
  177--182. \MR{1402516}

\end{thebibliography}

\vspace{0.5cm}

\normalsize

\noindent{\textsc{Department of Pure Mathematics and Mathematical Statistics, University of Cambridge, UK}}

\noindent{\textit{E-mail address:} \texttt{ff373@cam.ac.uk}}

\vspace{0.2cm}

\noindent{\textsc{Faculty of Mathematics, Karlsruhe Institute of Technology, Germany}}

\noindent{\textit{E-mail address:} \texttt{roman.sauer@kit.edu}}

\end{document}